\documentclass{article}
\usepackage{amsmath,amssymb,amsthm}

\usepackage[colorlinks=true,linkcolor=blue,citecolor=blue,pdfpagelabels=false]{hyperref}

\theoremstyle{plain}
  \newtheorem{theorem}{Theorem}[section]
  \newtheorem{lemma}[theorem]{Lemma}

  \newtheorem{proposition}[theorem]{Proposition}
  
\theoremstyle{definition}
  \newtheorem{example}[theorem]{Example}
  \newtheorem{remark}[theorem]{Remark}
  
\newenvironment{acknowledgements}{\bigskip\textbf{Acknowledgements.}}{}

\newcommand{\mathd}{\mathrm{d}}

\newcommand{\nin}{\not\in}

\renewcommand{\geq}{\geqslant}
\renewcommand{\leq}{\leqslant}

\makeatletter
\newcommand{\subjclass}[2][1991]{%
  \let\@oldtitle\@title%
  \gdef\@title{\@oldtitle\footnotetext{#1 \emph{Mathematics subject classification.} #2}}%
}
\newcommand{\keywords}[1]{%
  \let\@@oldtitle\@title%
  \gdef\@title{\@@oldtitle\footnotetext{\emph{Key words and phrases.} #1.}}%
}
\makeatother

\begin{document}

\title{Special values of trigonometric Dirichlet series and Eichler integrals}



\author{Armin Straub\thanks{Email: \texttt{astraub@illinois.edu}}\\
Department of Mathematics\\
University of Illinois at Urbana-Champaign}

\date{\today}

\maketitle

\begin{abstract}
  We provide a general theorem for evaluating trigonometric Dirichlet series
  of the form $\sum_{n \geq 1} \frac{f (\pi n \tau)}{n^s}$, where $f$ is
  an arbitrary product of the elementary trigonometric functions, $\tau$ a
  real quadratic irrationality and $s$ an integer of the appropriate parity.
  This unifies a number of evaluations considered by many authors, including
  Lerch, Ramanujan and Berndt. Our approach is based on relating the series to
  combinations of derivatives of Eichler integrals and polylogarithms.
\end{abstract}

\section{Introduction}

Special values of trigonometric Dirichlet series have been studied by Cauchy,
Lerch, Mellin, Hardy, Ramanujan, Watson and many others since the early 20th
century. Examples, which we will refer to later, include
\cite{berndt-dedekindsum76}, \cite{berndt-modular77},
\cite{berndt-series78}, \cite{kmt-barnes}, \cite{lrr-secantzeta},
\cite{bs-secantsums}, \cite{cg-secant}. Many further references,
especially to early publications, can be found in \cite[Chapter
14]{berndtII} and \cite[Chapter 37]{berndtV}. The long history of these
series includes, as an early example, the formulas
\begin{equation}
  \sum_{n = 1}^{\infty} \frac{\cot (\pi n i)}{n^{2 r - 1}} = \frac{1}{2} (2
  \pi i)^{2 r - 1} \sum_{m = 0}^r (- 1)^{m + 1} \frac{B_{2 m}}{(2 m) !} 
  \frac{B_{2 (r - m)}}{(2 (r - m)) !}, \label{eq:cot:i}
\end{equation}
for even $r$, which go back to Cauchy and Lerch with later proofs given by
several authors; see \cite[(6.2)]{berndt-dedekindsum76} or \cite[Entry
14.25]{berndtII} for a detailed history. In particular, the special case $r =
4$ in \eqref{eq:cot:i}, that is
\begin{equation}
  \sum_{n = 1}^{\infty} \frac{\cot (\pi n i)}{n^7} = - \frac{19 i}{56, 700}
  \pi^7, \label{eq:cot:i:7}
\end{equation}
was one of the formulas Ramanujan included in his first letter to Hardy
\cite[Entry 14.25(ii)]{berndtII}. These series are usually written in terms
of the hyperbolic cotangent (so that all quantities involved are real), but
the above forms are more natural for our purposes. The evaluations
\eqref{eq:cot:i} are a consequence of and are explained by Ramanujan's famous
and inspiring formula
\begin{eqnarray}
  &  & \alpha^{- m} \left\{ \frac{\zeta (2 m + 1)}{2} + \sum_{n = 1}^{\infty}
  \frac{n^{- 2 m - 1}}{e^{2 \alpha n} - 1} \right\} = (- \beta)^{- m} \left\{
  \frac{\zeta (2 m + 1)}{2} + \sum_{n = 1}^{\infty} \frac{n^{- 2 m - 1}}{e^{2
  \beta n} - 1} \right\} \nonumber\\
  &  & - 2^{2 m} \sum_{n = 0}^{m + 1} (- 1)^n \frac{B_{2 n}}{(2 n) !} 
  \frac{B_{2 m - 2 n + 2}}{(2 m - 2 n + 2) !} \alpha^{m - n + 1} \beta^n, 
  \label{eq:rama}
\end{eqnarray}
where $\alpha$ and $\beta$ are positive numbers with $\alpha \beta = \pi^2$
and $m$ is any nonzero integer. We refer to \cite{berndt-modular77} or
\cite[Entry 14.21(i)]{berndtII}, as well as the references therein. In
modern language, \eqref{eq:rama} can be seen to express the fact that, for odd
$s$, the cotangent Dirichlet series \cite{gun-murty-rath2011}
\begin{equation}
  \xi_s (\tau) = \sum_{n = 1}^{\infty} \frac{\cot (\pi n \tau)}{n^s}
  \label{eq:cot}
\end{equation}
is an Eichler integral of the Eisenstein series of weight $s + 1$ and level
$1$. We briefly review Eichler integrals in Section~\ref{sec:eichler}. As a
consequence, for certain $s$, $\xi_s (\tau)$ can be explicitly evaluated at
the values $\tau = i$ or $\tau = e^{2 \pi i / 3}$, which are fixed points of
linear fractional transformations induced by the modular group $\operatorname{SL}_2
(\mathbb{Z})$. Up to the action of $\operatorname{SL}_2 (\mathbb{Z})$, these are the
only fixed points in the upper half-plane. More generally, however, every real
quadratic irrationality occurs as the fixed point of some $\gamma \in
\operatorname{SL}_2 (\mathbb{Z})$. This is reviewed in Section~\ref{sec:bg}. Though
convergence of series such as \eqref{eq:cot} becomes an interesting issue when
$\tau$ is real, see Section~\ref{sec:conv}, one finds that $\xi_s (\rho) \in
\rho \pi^s \mathbb{Q}$ for any real quadratic irrational $\rho$ provided that
$s \geq 2$ is odd. For instance,
\begin{equation}
  \xi_3 ( \sqrt{7}) = \sum_{n = 1}^{\infty} \frac{\cot (\pi n \sqrt{7})}{n^3}
  = - \frac{\sqrt{7}}{20} \pi^3 . \label{eq:cot:eg7}
\end{equation}
Such evaluations of the cotangent Dirichlet series are discussed in
\cite{berndt-dedekindsum76}. See Example~\ref{eg:rama} for similar known
results.

As a recent addition to the zoo of special values of trigonometric Dirichlet
series, Y.~Komori, K.~Matsumoto and H.~Tsumura \cite{kmt-barnes}, based on
formulas for the Barnes multiple zeta-functions, discovered identities
including
\begin{equation}
  \sum_{n = 1}^{\infty} \frac{\cot^2 (\pi n \zeta_3)}{n^4} = - \frac{31}{2835}
  \pi^4, \hspace{1em} \hspace{1em} \sum_{n = 1}^{\infty} \frac{\csc^2 (\pi n
  \zeta_3)}{n^4} = \frac{1}{5670} \pi^4, \label{eq:cot2:kmt}
\end{equation}
where $\zeta_3 = e^{2 \pi i / 3}$ is the cube root of unity. One purpose of
and motivation for the present note is to put all these evaluations into a
general context.

Our main result shows that, depending on the parity of $s$, any trigonometric
Dirichlet series
\begin{equation*}
  \psi^{a, b}_s (\tau) = \sum_{n = 1}^{\infty} \frac{\operatorname{trig}^{a, b} (\pi
   n \tau)}{n^s}, \hspace{1em} \operatorname{trig}^{a, b} = \sec^a \csc^b,
\end{equation*}
where $a, b$ are integers, has the property that it can be evaluated when
$\tau$ is a real quadratic irrationality. Note that any product of the
trigonometric functions $\cos (x)$, $\sin (x)$, $\sec (x)$, $\csc (x)$, $\tan
(x)$, and $\cot (x)$ can be expressed as $\sec^a (x) \csc^b (x)$ for (unique)
integers $a$ and $b$. Our main result is the following.

\begin{theorem}
  \label{thm:eval}Let $\rho$ be a real quadratic irrationality, and let $a, b,
  s$ be integers such that, for convergence, $s \geq \max (a, b, 1) + 1$.
  If $s$ and $b$ have the same parity, then
  \begin{equation*}
    \psi^{a, b}_s (\rho) = \sum_{n = 1}^{\infty} \frac{\operatorname{trig}^{a, b}
     (\pi n \rho)}{n^s} \in \pi^s \mathbb{Q} (\rho) .
  \end{equation*}
  Moreover, if, in addition, $\rho^2 \in \mathbb{Q}$ and $a + b \geq 0$,
  then $\psi^{a, b}_s (\rho) \in (\pi \rho)^s \mathbb{Q}$.
\end{theorem}

The underlying reason for Theorem~\ref{thm:eval}, which we prove in
Section~\ref{sec:eval}, is that $\psi^{a, b}_s (\tau)$ can be expressed as a
linear combination of derivatives of Eichler integrals and polylogarithms.
That the series \ $\psi^{a, b}_s (\tau)$ indeed converges when $\tau$ is a
real algebraic irrationality and $s \geq \max (a, b, 1) + 1$ is proved in
Section~\ref{sec:conv}.

Note that Theorem~\ref{thm:eval} includes as a special case the recent
conjecture \cite{lrr-secantzeta} of M.~Lal{\'i}n, F.~Rodrigue and M.~Rogers
that, for even $s > 0$ and all rational $r > 0$, the values
\begin{equation*}
  \psi_s^{1, 0} ( \sqrt{r}) = \sum_{n = 1}^{\infty} \frac{\sec (\pi n
   \sqrt{r})}{n^s}
\end{equation*}
are rational multiples of $\pi^s$. Independent proofs of this conjecture have
been given by P.~Charollois and M.~Greenberg \cite{cg-secant} as well as
B.~Berndt and A.~Straub \cite{bs-secantsums}.

\begin{example}
  We record some random examples to illustrate Theorem~\ref{thm:eval}:
  \begin{eqnarray*}
    \sum_{n = 1}^{\infty} \frac{\sec^2 (\pi n \sqrt{5})}{n^4} & = &
    \frac{14}{135} \pi^4,\\
    \sum_{n = 1}^{\infty} \frac{\cot^2 (\pi n \sqrt{5})}{n^4} & = &
    \frac{13}{945} \pi^4,\\
    \sum_{n = 1}^{\infty} \frac{\csc^2 (\pi n \sqrt{11})}{n^4} & = &
    \frac{8}{385} \pi^4,\\
    \sum_{n = 1}^{\infty} \frac{\sec^3 (\pi n \sqrt{2})}{n^4} & = & -
    \frac{2483}{5220} \pi^4,\\
    \sum_{n = 1}^{\infty} \frac{\tan^3 (\pi n \sqrt{6})}{n^5} & = & \frac{35,
    159}{17, 820 \sqrt{6}} \pi^4 .
  \end{eqnarray*}
  These values have been obtained by tracing the proof of
  Theorem~\ref{thm:eval}, which provides a method to compute such evaluations.
  
  In addition, the evaluation
  \begin{equation*}
    \psi_3^{- 2, 1} ( \sqrt{2}) = \sum_{n = 1}^{\infty} \frac{(\cos \cot)
     (\pi n \sqrt{2})}{n^3} = \left[ \frac{1}{2} - \frac{253}{360 \sqrt{2}}
     \right] \pi^3
  \end{equation*}
  illustrates that the condition $a + b \geq 0$ is required for the last
  part of Theorem~\ref{thm:eval}.
\end{example}

\begin{remark}
  \label{rk:H:troubles}Theorem~\ref{thm:eval} is stated for real quadratic
  irrationalities $\rho$ only. Its proof, however, extends to certain nonreal
  $\rho$ which are the fixed points of linear fractional transformations
  $\gamma \in \operatorname{SL}_2 (\mathbb{Z})$. For instance, the evaluation of
  $\psi^{- 2, 2}_4 (\zeta_3)$ in equation \eqref{eq:cot2:kmt} can be achieved
  by our method because $\zeta_3$ is fixed by $S T$, defined in \eqref{eq:TS},
  and because the cotangent Dirichlet series is an Eichler integral for the
  full modular group. Similarly, formula \eqref{eq:cot:i:7}, and more
  generally \eqref{eq:cot:i}, follow by evaluating $\psi_{4 r - 1}^{- 1, 1}
  (i)$ as in the proof of Theorem~\ref{thm:eval}.
  
  There are, however, two complications for these (and other) nonreal values.
  Firstly, for instance, while $\zeta_3$ is fixed by $\pm S T$ and $\pm (S
  T)^2$, it is easily seen from \eqref{eq:fix} that $\zeta_3$ is not fixed by
  any other nontrivial linear fractional transformation. In particular, it is
  not fixed by any transformation in $\Gamma (2)$. In consequence, we cannot
  apply our approach to evaluate the secant Dirichlet series $\psi_s^{1, 0}
  (\zeta_3)$ for any $s$ (and, to our knowledge, no evaluation as an algebraic
  multiple of $\pi^s$ is known).
  
  Secondly, the quantity \eqref{eq:div0}, which we divide by, can be zero.
  This is the reason why we can evaluate $\psi^{- 2, 2}_s (\zeta_3)$ only for
  $s$ of the form $s = 6 r + 4$. These are exactly the cases for which these
  series are evaluated in \cite[Corollary 6.4]{kmt-barnes} by different
  means.
\end{remark}

\begin{example}
  \label{eg:rama}Besides \eqref{eq:cot} or \eqref{eq:cot2:kmt}, other types of
  natural trigonometric Dirichlet series have been considered by many authors.
  For instance, Ramanujan recorded
  \begin{equation*}
    \sum_{n = 0}^{\infty} \frac{\tanh ((2 n + 1) \pi / 2)}{(2 n + 1)^3} =
     \frac{\pi^3}{32}, \hspace{1em} \sum_{n = 1}^{\infty} \frac{(- 1)^{n + 1}
     \operatorname{csch} (\pi n)}{n^3} = \frac{\pi^3}{360},
  \end{equation*}
  as well as
  \begin{equation*}
    \sum_{n = 1}^{\infty} \frac{\chi (n) \operatorname{sech} (\pi n / 2)}{n^5} =
     \frac{\pi^5}{768},
  \end{equation*}
  where $\chi = ( \tfrac{- 4}{\cdot})$ denotes the nonprincipal Dirichlet
  character modulo $4$ (that is, $\chi (n) = 0$ for even $n$, and $\chi (n) =
  (- 1)^{(n - 1) / 2}$ for odd $n$). These formulas can be found in
  \cite[Entry 14.25]{berndtII} along with their history and generalizations.
  Since it might not be immediately obvious, let us indicate in
  \eqref{eq:csc:r}, \eqref{eq:tan:r} and \eqref{eq:sec:r} how these series
  relate to the series $\psi^{a, b}_s$ that we consider here. In particular,
  this demonstrates that the approach of Theorem~\ref{thm:eval} applies to
  establishing the corresponding evaluations given by Ramanujan.
  
  To begin with, since $\csc ( z + \pi n) = ( - 1)^n \csc ( z)$, we find
  \begin{equation}
    \sum_{n = 1}^{\infty} \frac{(- 1)^{n + 1} \csc (\pi n \tau)}{n^s} = -
    \psi^{0, 1}_s (\tau + 1) . \label{eq:csc:r}
  \end{equation}
  Next, note that $\cot (z + \frac{\pi n}{2})$ equals $\cot (z)$ if $n$ is
  even, and $- \tan (z)$ if $n$ is odd. Consequently,
  \begin{equation*}
    \sum_{n = 1}^{\infty} \frac{\cot (\pi n (\tau + \frac{1}{2}))}{n^s} =
     \frac{1}{2^s} \sum_{n = 1}^{\infty} \frac{\cot (2 \pi n \tau)}{n^s} -
     \sum_{n = 0}^{\infty} \frac{\tan (\pi (2 n + 1) \tau)}{(2 n + 1)^s},
  \end{equation*}
  or, equivalently,
  \begin{equation}
    \sum_{n = 0}^{\infty} \frac{\tan (\pi (2 n + 1) \tau)}{(2 n + 1)^s} =
    \frac{1}{2^s} \psi^{- 1, 1}_s ( 2 \tau) - \psi^{- 1, 1}_s \left( \tau +
    \tfrac{1}{2} \right) . \label{eq:tan:r}
  \end{equation}
  Explicit evaluations of the series \eqref{eq:tan:r} for certain real
  quadratic irrationalities $\tau$ have been obtained in \cite[Theorem
  4.11]{berndt-series78}. Theorem~\ref{thm:eval} shows, less explicitly, that
  such evaluations are possible for all real quadratic irrationalities $\tau$.
  
  Similarly, $\csc (z + \frac{\pi n}{2})$ equals $(- 1)^{n / 2} \csc (z)$ if
  $n$ is even, and $(- 1)^{(n - 1) / 2} \sec (z)$ if $n$ is odd. We thus
  conclude
  \begin{equation}
    \sum_{n = 1}^{\infty} \frac{\chi (n) \sec (\pi n \tau)}{n^s} = \psi_s^{0,
    1} \left( \tau + \tfrac{1}{2} \right) - \frac{1}{2^s} \psi_s^{0, 1} (2
    \tau + 1) . \label{eq:sec:r}
  \end{equation}
  For each of the series \eqref{eq:csc:r}, \eqref{eq:tan:r} and
  \eqref{eq:sec:r}, Theorem~\ref{thm:eval} proves that, depending on parity,
  they evaluate at real quadratic irrationalities $\tau$ as multiples of
  $\pi^s$. We conclude with some simple explicit examples:
  \begin{eqnarray*}
    \sum_{n = 1}^{\infty} \frac{(- 1)^{n + 1} \csc (\pi n \sqrt{13})}{n^3} & =
    & - \frac{\pi^3}{12 \sqrt{13}},\\
    \sum_{n = 0}^{\infty} \frac{\tan (\pi (2 n + 1) \sqrt{5})}{(2 n + 1)^5} &
    = & \frac{23 \pi^5}{3456 \sqrt{5}},\\
    \sum_{n = 1}^{\infty} \frac{\chi (n) \sec (\pi n \sqrt{7})}{n^3} & = & -
    \frac{7 \pi^3}{96} .
  \end{eqnarray*}
\end{example}

\section{Convergence}\label{sec:conv}

Let $a \geq 0$ and $b \geq 0$ and assume that at least one of them
is positive. Then the series
\begin{equation*}
  \psi^{a, b}_s (\tau) = \sum_{n = 1}^{\infty} \frac{(\sec^a \csc^b) (\pi n
   \tau)}{n^s}
\end{equation*}
converges absolutely for all nonreal $\tau$. For rational $\tau$, the series
$\psi^{a, b}_s (\tau)$ converges absolutely for $s > 1$ provided that all its
terms are finite; this requires $b \leq 0$ and, in addition, $\tau$ needs
to have odd denominator if $a > 0$. Convergence for real irrationalities
$\tau$, on the other hand, is a much more subtle question; see, for instance,
\cite{rivoal-conv12}.

It is shown in \cite[Theorem 1]{lrr-secantzeta} that the secant Dirichlet
series $\psi_s^{1, 0} (\tau)$ converges absolutely for algebraic irrational
$\tau$ whenever $s \geq 2$. While the case $s > 2$ follows from an
application of the Thue-Siegel-Roth Theorem, the case $s = 2$ requires a
rather subtle argument due to Florian Luca. The next result generalizes these
conclusions to $\psi_s^{a, b} (\tau)$ for any integers $a, b$. In addition, it
strengthens \cite[Theorem 5.1]{berndt-dedekindsum76}, which proves a weaker
result in the case $(a, b) = (- 1, 1)$.

\begin{theorem}
  Let $a, b$ be integers and $\tau$ real. The series $\psi^{a, b}_s (\tau)$
  converges absolutely
  \begin{enumerate}
    \item for $s > 1$, if $a \leq 0$ and $b \leq 0$;
    
    \item for $s \geq \max (a, b) + 1$, if $\tau$ is algebraic irrational
    and $\max (a, b) > 0$.
  \end{enumerate}
\end{theorem}

\begin{proof}
  The first part is obvious because cosine and sine are bounded on the real
  line, so that, for $a \leq 0$ and $b \leq 0$, $\psi^{a, b}_s
  (\tau)$ can be bounded from above by the Riemann zeta function $\zeta (s)$.
  
  For the second part, note that $\sec^2 (z) \csc^2 (z) = \sec^2 (z) + \csc^2
  (z)$ implies the simple reduction identity
  \begin{equation}
    \operatorname{trig}^{a, b} (z) = \operatorname{trig}^{a - 2, b} (z) + \operatorname{trig}^{a, b -
    2} (z) . \label{eq:trig:red2}
  \end{equation}
  Applying \eqref{eq:trig:red2} recursively and again using boundedness of
  cosine and sine, our claim follows if we can show that, for $\lambda > 0$
  and algebraic irrational $\tau$, the series
  \begin{equation*}
    \psi^{\lambda, 0}_s (\tau) = \sum_{n = 1}^{\infty} \frac{\sec^{\lambda}
     (\pi n \tau)}{n^s}, \hspace{1em} \psi^{0, \lambda}_s (\tau) = \sum_{n =
     1}^{\infty} \frac{\csc^{\lambda} (\pi n \tau)}{n^s}
  \end{equation*}
  converge absolutely whenever $s \geq \lambda + 1$. These claims are
  proved in Lemmas~\ref{lem:conv:csc} and \ref{lem:conv:sec} below.
\end{proof}

In the same manner as in \cite{lrr-secantzeta}, we will use the following
weak version of a result due to Worley \cite{worley81}. Here and in the
sequel, $p_n / q_n$ denotes the $n$th convergent of the continued fraction
expansion $[a_0 ; a_1, a_2, \ldots]$ of $\tau$. It is well-known that $p_n$
and $q_n$ satisfy
\begin{equation}
  p_n = a_n p_{n - 1} + p_{n - 2}, \hspace{1em} q_n = a_n q_{n - 1} + q_{n -
  2} . \label{eq:pq:rec}
\end{equation}
\begin{theorem}
  \label{thm:worley}Let $\tau$ be irrational, $k > \frac{1}{2}$, and $p / q$ a
  rational approximation to $\tau$ in reduced form for which
  \begin{equation*}
    \left| \tau - \frac{p}{q} \right| < \frac{k}{q^2} .
  \end{equation*}
  Then $p / q$ is of the form
  \begin{equation*}
    \frac{p}{q} = \frac{a p_m + b p_{m - 1}}{a q_m + b q_{m - 1}},
     \hspace{1em} |a|, |b| < 2 k,
  \end{equation*}
  where $a$ and $b$ are integers, and $m$ is the largest index up to which the
  continued fractions of $\tau$ and $p / q$ agree.
\end{theorem}

The next result is proved by a natural extension of the proof of
\cite[Theorem 1]{lrr-secantzeta}, which is due to Florian Luca. As indicated
in Remark~\ref{rk:conv:csc}, absolute convergence of $\sum_{n = 1}^{\infty}
\frac{\csc^{\lambda} (\pi n \tau)}{n^s}$ for $s > \lambda + 1$ is much simpler
to deduce.

\begin{lemma}
  \label{lem:conv:csc}Let $\lambda > 0$ and $\tau$ be algebraic irrational.
  Then the series $\psi^{0, \lambda}_{\lambda + 1} (\tau) = \sum_{n =
  1}^{\infty} \frac{\csc^{\lambda} (\pi n \tau)}{n^{\lambda + 1}}$ converges
  absolutely.
\end{lemma}

\begin{proof}
  Starting with the elementary
  \begin{equation*}
    | \sin (\pi \tau) | \geq | \tau - k|,
  \end{equation*}
  where $k = [\tau]$ is the nearest integer to $\tau$, we obtain
  \begin{equation*}
    \frac{| \csc (\pi n \tau) |^{\lambda}}{n^{\lambda + 1}} \leq
     \frac{1}{n^{\lambda + 1} |n \tau - k_n |^{\lambda}} = \frac{1}{n^{2
     \lambda + 1} | \tau - k_n / n|^{\lambda}}
  \end{equation*}
  with $k_n = [n \tau]$, which is the integer maximizing the right-hand side.
  We first consider those indices $n$ for which the right-hand side is
  sufficiently small. Indeed, we notice that our series restricted to the
  indices in the set
  \begin{equation*}
    W_{\tau} = \left\{ n \geq 0 : \hspace{1em} \left| \tau -
     \frac{k_n}{n} \right| \geq \frac{(\log n)^{\alpha}}{n^2} \right\}
  \end{equation*}
  is easily seen to converge when we choose $\alpha$ large enough; namely,
  \begin{equation*}
    \sum_{n \in W_{\tau}} \frac{| \csc (\pi n \tau) |^{\lambda}}{n^{\lambda +
     1}} \leq \sum_{n \in W_{\tau}} \frac{1}{n^{2 \lambda + 1} | \tau -
     k_n / n|^{\lambda}} \leq \sum_{n \in W_{\tau}} \frac{1}{n (\log
     n)^{\alpha \lambda}} < \infty
  \end{equation*}
  provided that $\alpha \lambda > 1$. In the sequel, we assume that $\alpha$
  has been chosen such that $\alpha \lambda > 1$.
  
  On the other hand, assume that $n \nin W_{\tau}$, in which case
  \begin{equation}
    \left| \tau - \frac{k_n}{n} \right| \leq \frac{(\log
    n)^{\alpha}}{n^2} . \label{eq:nnotinW}
  \end{equation}
  Let $p_m / q_m$ be the convergents of $\tau$ and let $\ell$ be such that $n
  < q_{\ell}$. Then Worley's Theorem~\ref{thm:worley} applied with $k = (\log
  q_{\ell})^{\alpha} / d^2$, where $d = (k_n, n)$, shows that
  \begin{equation*}
    \frac{k_n}{n} = \frac{a p_m + b p_{m - 1}}{a q_m + b q_{m - 1}}
  \end{equation*}
  where $m < \ell$ and $a, b$ are integers with $|a|, |b| < 2 (\log
  q_{\ell})^{\alpha} / d^2$. In particular, $n = d (a q_m + b q_{m - 1}) = r
  q_m + s q_{m - 1},$ where $m < \ell$ and $r, s$ are integers with $|r|, |s|
  < 2 (\log q_{\ell})^{\alpha}$. We conclude that there can be at most $16
  \ell (\log q_{\ell})^{2 \alpha}$ values of $n$ less than $q_{\ell}$ for
  which \eqref{eq:nnotinW} holds.
  
  Since $\tau$ is algebraic irrational, the Thue--Siegel--Roth Theorem implies
  that, for each $\varepsilon > 0$, there is a constant $C (\tau,
  \varepsilon)$ such that
  \begin{equation}
    \left| \tau - \frac{p}{q} \right| > \frac{C (\tau, \varepsilon)}{q^{2 +
    \varepsilon}} \label{eq:tsr}
  \end{equation}
  for all fractions $p / q$. On the other hand, we have
  \begin{equation*}
    \frac{1}{q_{\ell} q_{\ell + 1}} > \left| \tau - \frac{p_{\ell}}{q_{\ell}}
     \right|,
  \end{equation*}
  which combined with the Thue--Siegel--Roth Theorem shows that $q_{\ell + 1}
  < C (\tau, \varepsilon) q_{\ell}^{1 + \varepsilon}$. Because the convergents
  $p_{\ell} / q_{\ell}$ of continued fractions provide best possible
  approximations to $\tau$ among fractions with denominator at most
  $q_{\ell}$, we find that, for $n < q_{\ell}$,
  \begin{equation*}
    \left| \tau - \frac{k_n}{n} \right| > \left| \tau -
     \frac{p_{\ell}}{q_{\ell}} \right| > \frac{C (\tau,
     \varepsilon)}{q_{\ell}^{2 + \varepsilon}} .
  \end{equation*}
  Assuming, in addition, $n \geq q_{\ell - 1}$, we thus obtain
  \begin{equation*}
    \frac{| \csc (\pi n \tau) |^{\lambda}}{n^{\lambda + 1}} \leq
     \frac{1}{n^{2 \lambda + 1} | \tau - k_n / n|^{\lambda}} <
     \frac{q_{\ell}^{(2 + \varepsilon) \lambda}}{C (\tau,
     \varepsilon)^{\lambda} q_{\ell - 1}^{2 \lambda + 1}} < \frac{C (\tau,
     \varepsilon)^{(1 + \varepsilon) \lambda}}{q_{\ell - 1}^{1 - \varepsilon
     (3 + \varepsilon) \lambda}} = \frac{C (\tau, \varepsilon')}{q_{\ell -
     1}^{1 - \varepsilon'}},
  \end{equation*}
  where $\varepsilon' = \varepsilon (3 + \varepsilon) \lambda > 0$.
  
  Combining our observations, we have
  \begin{equation*}
    \sum_{n \nin W_{\tau}} \frac{| \csc (\pi n \tau) |^{\lambda}}{n^{\lambda
     + 1}} \leq \sum_{\ell = 1}^{\infty} 16 \ell (\log q_{\ell})^{2
     \alpha}  \frac{C (\tau, \varepsilon')}{q_{\ell - 1}^{1 - \varepsilon'}},
  \end{equation*}
  and convergence follows from the fact that, by comparison with the Fibonacci
  numbers $F_{\ell}$ via \eqref{eq:pq:rec}, the sequence $q_{\ell}$ grows at
  least as fast as $\varphi^{\ell}$ with $\varphi = (1 + \sqrt{5}) / 2$.
\end{proof}

\begin{remark}
  \label{rk:conv:csc}With $\lambda$ and $\tau$ as in Lemma~\ref{lem:conv:csc},
  absolute convergence of the series
  \begin{equation*}
    \sum_{n = 1}^{\infty} \frac{\csc^{\lambda} (\pi n \tau)}{n^{\lambda + 1 +
     \delta}},
  \end{equation*}
  for $\delta > 0$, is much simpler to deduce. Indeed, estimating as in the
  proof of Lemma~\ref{lem:conv:csc}, we find
  \begin{equation*}
    \frac{| \csc (\pi n \tau) |^{\lambda}}{n^{\lambda + 1 + \delta}}
     \leq \frac{1}{n^{2 \lambda + 1 + \delta} | \tau - k_n /
     n|^{\lambda}} < \frac{C (\tau, \varepsilon)^{- \lambda}}{n^{1 + \delta -
     \lambda \varepsilon}},
  \end{equation*}
  where $C (\tau, \varepsilon)$ is the constant from applying the
  Thue--Siegel--Roth Theorem \eqref{eq:tsr}. Convergence follows upon choosing
  $\varepsilon$ such that $\lambda \varepsilon < \delta$.
\end{remark}

\begin{lemma}
  \label{lem:conv:sec}Let $\lambda > 0$ and $\tau$ be algebraic irrational.
  Then the series $\psi^{\lambda, 0}_{\lambda + 1} (\tau) = \sum_{n =
  1}^{\infty} \frac{\sec^{\lambda} (\pi n \tau)}{n^{\lambda + 1}}$ converges
  absolutely.
\end{lemma}

\begin{proof}
  The proof proceeds along similar lines as the proof of
  Lemma~\ref{lem:conv:csc}. Indeed, we have the elementary relation
  \begin{equation*}
    \sec ( \tau) = \csc \left( \tau + \frac{\pi}{2} \right),
  \end{equation*}
  so that
  \begin{equation*}
    \frac{| \sec (\pi n \tau) |^{\lambda}}{n^{\lambda + 1}} = \frac{| \csc
     \left( \pi (n \tau + \tfrac{1}{2}) \right) |^{\lambda}}{n^{\lambda + 1}}
     \leq \frac{1}{n^{\lambda + 1} |n \tau + \tfrac{1}{2} - k_n
     |^{\lambda}} = \frac{1}{n^{2 \lambda + 1} \left| \tau - \tfrac{2 k_n -
     1}{2 n} \right|^{\lambda}}
  \end{equation*}
  with $k_n = [n \tau + \tfrac{1}{2}]$. It remains to argue as in the proof of
  Lemma~\ref{lem:conv:csc} and we omit the details.
\end{proof}

\section{Background on fractional linear transformations}\label{sec:bg}

We denote with $T$, $S$ and $R$ the matrices
\begin{equation}
  T = \left(\begin{array}{cc}
    1 & 1\\
    0 & 1
  \end{array}\right), \hspace{1em} S = \left(\begin{array}{cc}
    0 & - 1\\
    1 & 0
  \end{array}\right), \hspace{1em} R = \left(\begin{array}{cc}
    1 & 0\\
    1 & 1
  \end{array}\right), \label{eq:TS}
\end{equation}
and recall that the matrices $T$ and $S$ generate $\Gamma_1 = \operatorname{SL}_2
(\mathbb{Z})$. The principal congruence subgroup $\Gamma (N)$ of $\Gamma_1$
consists of those matrices that are congruent to the identity matrix $I$
modulo $N$. More generally, congruence subgroups of $\Gamma_1$ are those
subgroups $\Gamma \leq \Gamma_1$ containing $\Gamma (N)$ for some $N$;
the minimal such $N$ being the level of $\Gamma$.

As usual, we consider the action of $\Gamma_1$ on complex numbers $\tau$ by
fractional linear transformations and write
\begin{equation*}
  \left(\begin{array}{cc}
     a & b\\
     c & d
   \end{array}\right) \cdot \tau = \frac{a \tau + b}{c \tau + d} .
\end{equation*}
Correspondingly, $\Gamma_1$ acts on the space of functions (on the upper
half-plane or on the full complex plane) via the slash operators $|_k$,
defined by
\begin{equation*}
  (f|_k \gamma) (\tau) = (c \tau + d)^{- k} f (\gamma \tau), \hspace{1em}
   \gamma = \left(\begin{array}{cc}
     a & b\\
     c & d
   \end{array}\right) \in \Gamma_1 .
\end{equation*}
This action extends naturally to the group algebra $\mathbb{C} [\Gamma_1]$.

Given a quadratic irrationality $\tau$, let $A x^2 + B x + C$, with $A > 0$
and $(A, B, C) = 1$, be its minimal polynomial and $\Delta = B^2 - 4 A C$ its
discriminant. We follow the exposition of \cite[p.~72]{zagier-intro123} and
observe that the fractional linear transformation
\begin{equation*}
  \gamma = \left(\begin{array}{cc}
     a & b\\
     c & d
   \end{array}\right) \in \operatorname{SL}_2 (\mathbb{Z}) .
\end{equation*}
fixes $\tau$ if and only if $(c, d - a, - b) = u (A, B, C)$ for some integral
factor of proportionality $u$. In that case, writing $t = a + d$ for the trace
of $\gamma$, we have
\begin{equation}
  \gamma = \left(\begin{array}{cc}
    \frac{t - B u}{2} & - C u\\
    A u & \frac{t + B u}{2}
  \end{array}\right), \hspace{4em} \det (\gamma) = \frac{t^2 - \Delta u^2}{4}
  . \label{eq:fix}
\end{equation}
Let us now restrict to real quadratic irrationalities $\tau$. In that case,
$\Delta > 0$. The above argument demonstrates that, if $t, u$ are solutions to
the Pell equation
\begin{equation*}
  t^2 - \Delta u^2 = 4,
\end{equation*}
then the fractional linear transformation $\gamma$ given in \eqref{eq:fix}
fixes $\tau$ (and all fractional linear transformations fixing $\tau$ arise
that way).

\begin{lemma}
  \label{lem:fix}Let $\tau$ be a quadratic irrationality, and $\Gamma
  \leq \Gamma_1$ a congruence subgroup. Then there exists $\gamma \in
  \Gamma$, with $\gamma \neq \pm I$, such that $\gamma \cdot \tau = \tau$.
\end{lemma}

\begin{proof}
  Let $N$ be such that $\Gamma (N) \leq \Gamma$. For every positive
  nonsquare $k$, Pell's equation
  \begin{equation}
    X^2 - k Y^2 = 1 \label{eq:pell}
  \end{equation}
  has nontrivial solutions $X$, $Y$. A proof of this fact was first published
  by Lagrange in 1768 \cite{lagrange-pell}, and we refer to
  \cite{lenstra-pell} for further information and background. As before, let
  $\Delta$ be the discriminant of $\tau$. Then $k = \Delta N^2$ is positive
  and not a perfect square, so that we find integers $X, Y$, with $Y \neq 0$,
  solving \eqref{eq:pell}. In light of the above discussion, setting $t = 2 X$
  and $u = 2 N Y$ in \eqref{eq:fix} gives a fractional linear transformation
  $\gamma$ which fixes $\tau$. Clearly, $\gamma \in \Gamma (N)$ and $\gamma
  \neq \pm I$.
\end{proof}

\section{Eichler integrals}\label{sec:eichler}

If $f (\tau)$ is a modular form of weight $k$ with respect to $\Gamma
\leq \Gamma_1$, then any $(k - 1)$st antiderivative of $f (\tau)$ is
called an {\emph{Eichler integral}}. Such an Eichler integral $F (\tau)$ is
characterized by the property that, for any $\gamma \in \Gamma$, $F|_{2 - k}
[\gamma - 1]$ is a polynomial of degree at most $k - 2$. These are referred to
as the {\emph{period polynomials}} of $f$ and their coefficients encode the
critical $L$-values of $f$. For the general theory of period polynomials we
refer to \cite{popa-period} and the references therein.

As mentioned in the introduction, Ramanujan's formula \eqref{eq:rama}
expresses the fact that, for odd $s$, the cotangent Dirichlet series $\xi_s
(\tau)$, defined in \eqref{eq:cot}, is, essentially, an Eichler integral.
Indeed, \eqref{eq:rama} may be expressed as
\begin{equation*}
  \xi_{2 m - 1} |_{2 - 2 m} [S - 1] = (- 1)^m (2 \pi)^{2 m - 1} \sum_{n =
   0}^m \frac{B_{2 n}}{(2 n) !}  \frac{B_{2 m - 2 n}}{(2 m - 2 n) !} \tau^{2 n
   - 1} .
\end{equation*}
The reason that $\xi_{2 m - 1} |_{2 - 2 m} [S - 1]$ are rational functions,
instead of polynomials, is that the $s$th derivative of $\xi_s (\tau)$ is an
Eisenstein series with the constant term of its Fourier expansion missing. We
refer to \cite{gun-murty-rath2011} for more details on the cotangent
Dirichlet series.

Similarly, it was shown in \cite{lrr-secantzeta} and \cite{bs-secantsums}
that, for even $s$, the secant Dirichlet series
\begin{equation*}
  \psi^{1, 0}_s (\tau) = \sum_{n = 1}^{\infty} \frac{\sec (\pi n \tau)}{n^s}
\end{equation*}
is, essentially, an Eichler integral of weight $1 - s$ with respect to the
modular group $\Gamma_2 = \langle T^2, R^2 \rangle$ generated by the matrices
\begin{equation}
  T^2 = \left(\begin{array}{cc}
    1 & 2\\
    0 & 1
  \end{array}\right), \hspace{1em} R^2 = \left(\begin{array}{cc}
    1 & 0\\
    2 & 1
  \end{array}\right) . \label{eq:AB}
\end{equation}
In other words, $\Gamma_2 \leq \Gamma ( 2)$ is the subgroup of the
principal modular subgroup $\Gamma ( 2)$ consisting of those matrices whose
diagonal entries are congruent to $1$ modulo $4$. More precisely, for any
$\gamma \in \Gamma_2$,
\begin{equation}
  \psi^{1, 0}_s |_{1 - s} [ \gamma - 1] = \pi^s p_s ( \gamma ; \tau),
  \label{eq:sec:pp}
\end{equation}
where $p_s ( \gamma ; \tau)$ is a rational function in $\tau$ with rational
coefficients. To be explicit, we have
\begin{eqnarray*}
  p_s ( T^2 ; \tau) & = & 0,\\
  p_s ( R^2 ; \tau) & = & [ z^{s - 1}]  \frac{\sin (\tau z)}{\sin (z) \sin ((2
  \tau + 1) z)},
\end{eqnarray*}
from which $p_s ( \gamma ; \tau)$ can be derived recursively in light of the
cocycle relation
\begin{equation*}
  p_s ( \alpha \beta ; \tau) = p_s ( \alpha ; \tau) |_{1 - s} \beta + p_s (
   \beta ; \tau) .
\end{equation*}
Several alternative expressions for $p_s ( R^2 ; \tau)$, for instance as
convolution sums involving Bernoulli numbers, are given in
\cite{lrr-secantzeta} and \cite{bs-secantsums}.

\section{Evaluating trigonometric Dirichlet series}\label{sec:eval}

The goal of this section is to prove Theorem~\ref{thm:eval}. Let us begin by
first considering the case $a \leq 0$ and $b \leq 0$, which is much
simpler and of a rather different nature than the other cases. Indeed, in that
case $\psi^{a, b}_s (\tau)$, with $s$ of the same parity as $b$, is piecewise
polynomial in $\tau$.

\begin{lemma}
  \label{lem:eval:cs}Let $\tau$ be real, and let $a, b, s$ be integers such
  that $a \leq 0$, $b \leq 0$ and $s > 1$. If $s$ and $b$ have the
  same parity, then
  \begin{equation*}
    \psi^{a, b}_s (\tau) = \sum_{n = 1}^{\infty} \frac{\operatorname{trig}^{a, b}
     (\pi n \tau)}{n^s} = \pi^s f (\tau),
  \end{equation*}
  where $f (\tau)$ is piecewise polynomial in $\tau$ with rational
  coefficients on each piece.
\end{lemma}

\begin{proof}
  Recall the classical formulas, valid for $0 < \tau < 1$,
  \begin{eqnarray}
    \sum_{n = 1}^{\infty} \frac{\cos (2 \pi n \tau)}{n^{2 m}} & = & \frac{(-
    1)^{m + 1}}{2} \frac{(2 \pi)^{2 m}}{(2 m) !} B_{2 m} (\tau), 
    \label{eq:ds:cos}\\
    \sum_{n = 1}^{\infty} \frac{\sin (2 \pi n \tau)}{n^{2 m + 1}} & = &
    \frac{(- 1)^{m + 1}}{2} \frac{(2 \pi)^{2 m + 1}}{(2 m + 1) !} B_{2 m + 1}
    (\tau),  \label{eq:ds:sin}
  \end{eqnarray}
  which may be found, for instance, in \cite[Theorem 3.2]{berndt-modular77}
  or \cite[Section 7.5.3]{lewin2}. Here, $B_n (x)$ denotes the $n$th
  Bernoulli polynomial. Note that $\operatorname{trig}_{a, b} (\tau)$ is an even or
  odd function depending on whether $b$ is even or odd. Writing
  $\operatorname{trig}_{a, b} (\tau)$ as a Fourier cosine or Fourier sine series,
  depending on the parity of $b$, we may express $\psi^{a, b}_s (\tau)$ as a
  finite linear combination of series of the type \eqref{eq:ds:cos} or
  \eqref{eq:ds:sin}. This shows that $\psi^{a, b}_s (\tau)$ is indeed $\pi^s$
  times a piecewise polynomial in $\tau$ with rational coefficients.
\end{proof}

The next result is a crucial building block for our proof of
Theorem~\ref{thm:eval}. Note that $\psi^{1, 0}_s$ is the secant Dirichlet
series.

\begin{proposition}
  \label{prop:eval:secD}Let $\psi_s = \psi^{1, 0}_s$. Let $\rho$ a real
  quadratic irrationality, and $j, s$ nonnegative integers with $s \geq 2
  j + 2$. If $j$ and $s$ have the same parity, then
  \begin{equation*}
    ( D^j \psi_s) ( \rho) \in \pi^s \mathbb{Q} ( \rho) .
  \end{equation*}
  If, in addition, $\rho^2 \in \mathbb{Q}$, then $( D^j \psi_s) ( \rho) \in (
  \pi \rho)^s \mathbb{Q}$.
\end{proposition}

\begin{proof}
  Observe that the statement for $j = 0$ has been proven in
  \cite{bs-secantsums} and follows from \eqref{eq:sec:pp}. We will prove the
  general statement by induction on $j$. In preparation, we first make the
  effect of differentiation on period functions explicit.
  
  Given a function $F$, integer $k$ and $\gamma = \left(\begin{array}{cc}
    a & b\\
    c & d
  \end{array}\right) \in \Gamma ( 1)$, denote with $p_{F, k} ( \gamma ; \tau)$
  the function
  \begin{equation*}
    p_{F, k} ( \gamma ; \tau) = F|_k [ \gamma - 1] ( \tau) = (c \tau + d)^{-
     k} F \left( \frac{a \tau + b}{c \tau + d} \right) - F (\tau) .
  \end{equation*}
  If $F$ is an Eichler integral of weight $k$, which transforms with respect
  to $\gamma$, then $p_{F, k} ( \gamma ; \tau)$ is one of its period
  polynomials. By differentiating both sides with respect to $\tau$, we find
  \begin{eqnarray*}
    D p_{F, k} ( \gamma ; \tau) & = & - c k (c \tau + d)^{- k - 1} F ( \gamma
    \tau) + (c \tau + d)^{- k - 2} ( D F) ( \gamma \tau) - ( D F) (\tau)\\
    & = & ( D F) |_{k + 2} [ \gamma - 1] ( \tau) - \frac{c k}{c \tau + d}
    F|_k \gamma\\
    & = & p_{D F, k + 2} ( \gamma ; \tau) - \frac{c k}{c \tau + d} ( F (
    \tau) + p_{F, k} ( \gamma ; \tau)),
  \end{eqnarray*}
  and hence
  \begin{equation}
    p_{D F, k + 2} ( \gamma ; \tau) = \frac{c k}{c \tau + d} (F ( \tau) +
    p_{F, k} ( \gamma ; \tau)) + D p_{F, k} ( \gamma ; \tau) . \label{eq:ppD}
  \end{equation}
  We now apply this observation in the case $F = \psi_s$ and $k = 1 - s$.
  Assuming that $s$ is even and $\gamma \in \Gamma_2$, equation
  \eqref{eq:sec:pp} shows that $p_{\psi_s, 1 - s} ( \gamma ; \tau) \in \pi^s
  \mathbb{Q} ( \tau)$. Inductively applying $( \ref{eq:ppD})$, we then find
  that
  \begin{equation*}
    (c \tau + d)^{s - 2 j - 1} ( D^j \psi_s) ( \gamma \tau) - ( D^j \psi_s)
     (\tau) = \pi^s f ( \tau) + \sum_{m = 0}^{j - 1} f_m ( \tau) ( D^m \psi_s)
     ( \tau),
  \end{equation*}
  where $f ( \tau), f_0 ( \tau), \ldots, f_{j - 1} ( \tau)$ are rational
  functions in $\tau$ with rational coefficients (these functions, of course,
  depend on $j$ and $s$). By Lemma~\ref{lem:fix}, there exists $\gamma \in
  \Gamma_2 \geq \Gamma (4)$ such that $\rho$ is fixed by $\gamma$.
  Consequently, we obtain
  \begin{equation}
    ( D^j \psi_s) (\rho) = \pi^s g ( \rho) + \sum_{m = 0}^{j - 1} g_m ( \rho)
    ( D^m \psi_s) ( \rho), \label{eq:DjPsiRho}
  \end{equation}
  where the rational functions $g ( \tau), g_0 ( \tau), \ldots, g_{j - 1} (
  \tau)$ are obtained from $f ( \tau), f_0 ( \tau), \ldots, f_{j - 1} ( \tau)$
  by dividing by
  \begin{equation}
    (c \tau + d)^{s - 2 j - 1} - 1. \label{eq:div0}
  \end{equation}
  It is important to note that the condition $s \geq 2 j + 2$ guarantees
  \eqref{eq:div0} to be nonzero when $\tau$ is a real quadratic irrationality.
  The first claim, that is $( D^j \psi_s) ( \rho) \in \pi^s \mathbb{Q} (
  \rho)$, now follows by induction on $j$.
  
  For the second claim, suppose that $\rho = \sqrt{r}$, where $r \in
  \mathbb{Q}$. Being algebraic conjugates, $- \rho$ is fixed by $\gamma$ as
  well so that \eqref{eq:DjPsiRho} also holds with $\rho$ replaced by $-
  \rho$. Also, observe that $( D^m \psi_s) ( \tau) / \tau^s$ is an even
  function of $\tau$. By induction, we may assume that, for $m = 0, 1, \ldots,
  j - 1$,
  \begin{equation*}
    \frac{( D^m \psi_s) ( \rho)}{\rho^s} = \frac{( D^m \psi_s) ( - \rho)}{( -
     \rho)^s} \in \pi^s \mathbb{Q}.
  \end{equation*}
  In combination with equation \eqref{eq:DjPsiRho}, we thus find that
  \begin{equation*}
    \frac{( D^j \psi_s) (\rho)}{\rho^s} = \pi^s h ( \rho), \hspace{1em}
     \frac{( D^j \psi_s) (- \rho)}{( - \rho)^s} = \pi^s h ( - \rho),
  \end{equation*}
  where $h ( \tau) \in \mathbb{Q} ( \tau)$ is a (single) rational function
  with rational coefficients. Since the left-hand sides are equal, we conclude
  that $h ( \rho) = h ( - \rho)$. The rationality of $h ( \tau)$ then implies
  that $h ( \rho) \in \mathbb{Q}$. This proves the claim.
\end{proof}

We next observe that the previous result for the secant Dirichlet series
$\psi^{1, 0}_s$ carries over to the cases of the cosecant Dirichlet series
$\psi^{0, 1}_s$, the cotangent Dirichlet series $\psi^{- 1, 1}_s$, and the
tangent Dirichlet series $\psi^{1, - 1}_s$.

\begin{proposition}
  \label{prop:eval:D}Let $(a, b)$ be one of $(1, 0)$, $(0, 1)$, $(- 1, 1)$,
  $(1, - 1)$. Let $\rho$ be a real quadratic irrationality, and $j, s$
  nonnegative integers with $s \geq 2 j + 2$. If $j + b$ and $s$ have the
  same parity, then
  \begin{equation*}
    ( D^j \psi_s^{a, b}) ( \rho) \in \pi^s \mathbb{Q} ( \rho) .
  \end{equation*}
  If, in addition, $\rho^2 \in \mathbb{Q}$, then $( D^j \psi_s^{a, b}) (
  \rho) \in ( \pi \rho)^s \mathbb{Q}$.
\end{proposition}

\begin{proof}
  The case $(a, b) = (1, 0)$ was proved in Proposition~\ref{prop:eval:secD}
  using the fact that, for $s$ of the required parity (namely, $s$ even),
  $\psi^{1, 0}_s$ is an Eichler integral of weight $1 - s$ with respect to the
  modular group $\Gamma_2$, whose period functions are $\pi^s$ times a
  rational function with rational coefficients.
  
  Recall that, for odd $s$, the cotangent Dirichlet series $\psi_s^{- 1, 1}$
  is an Eichler integral with respect to the full modular group. Indeed,
  Ramanujan's formula \eqref{eq:rama} shows that, for $s = 2 m - 1$,
  \begin{eqnarray*}
    \psi_s^{- 1, 1} |_{1 - s} [S - 1] & = & (- 1)^m (2 \pi)^s \sum_{n = 0}^m
    \frac{B_{2 n}}{(2 n) !}  \frac{B_{2 (m - n)}}{(2 (m - n)) !} \tau^{2 n -
    1},\\
    \psi_s^{- 1, 1} |_{1 - s} [T - 1] & = & 0.
  \end{eqnarray*}
  Consequently, the period functions are again $\pi^s$ times a rational
  function with rational coefficients. The proof of
  Proposition~\ref{prop:eval:secD} therefore applies to show the case $(a, b)
  = (- 1, 1)$ as well.
  
  Finally, note that the trigonometric relation $\csc (z) = \cot (z / 2) -
  \cot (z)$ implies that, for odd $s$, the cosecant Dirichlet series
  $\psi_s^{0, 1}$ is an Eichler integral with respect to $\Gamma^0 (2)$, the
  congruence subgroup of $\Gamma_1$ consisting of those matrices whose
  upper-right entry is even. Likewise, $\tan (z) = \cot (z) - 2 \cot (2 z)$
  shows that, for odd $s$, the tangent Dirichlet series $\psi_s^{1, - 1}$ is
  an Eichler integral with respect to $\Gamma_0 (2)$, consisting of those
  matrices in $\Gamma_1$ whose lower-left entry is even. Again, the period
  functions are of the required form to apply the proof of
  Proposition~\ref{prop:eval:secD} also in these two final cases.
\end{proof}

We are finally prepared to prove Theorem~\ref{thm:eval}, which is restated
below for the convenience of the reader.

\begin{theorem}
  Let $\rho$ be a real quadratic irrationality, and let $a, b, s$ be integers
  such that, for convergence, $s \geq \max (a, b, 1) + 1$. If $s$ and $b$
  have the same parity, then
  \begin{equation*}
    \psi^{a, b}_s (\rho) = \sum_{n = 1}^{\infty} \frac{\operatorname{trig}^{a, b}
     (\pi n \rho)}{n^s} \in \pi^s \mathbb{Q} (\rho) .
  \end{equation*}
  Moreover, if, in addition, $\rho^2 \in \mathbb{Q}$ and $a + b \geq 0$,
  then $\psi^{a, b}_s (\rho) \in (\pi \rho)^s \mathbb{Q}$.
\end{theorem}

\begin{proof}
  The case $a \leq 0$ and $b \leq 0$ follows as a special case of
  Lemma~\ref{lem:eval:cs}.
  
  Next, consider the case $a > 0$ and $b > 0$. By applying
  \eqref{eq:trig:red2}, that is $\psi^{a, b}_s = \psi_s^{a - 2, b} +
  \psi_s^{a, b - 2}$, recursively, we find that the general case follows if we
  can evaluate the cases $(a, b)$ with $a \in \{- 1, 0\}$ and $b > 0$ as well
  as the cases $(a, b)$ with $b \in \{- 1, 0\}$ and $a > 0$.
  
  In the remaining cases, we have either $a > 0$ and $b \leq 0$, or $a
  \leq 0$ and $b > 0$. If $a < - 1$, then we apply $\psi_s^{a, b} =
  \psi_s^{a + 2, b} - \psi_s^{a + 2, b - 2}$, which follows from
  \eqref{eq:trig:red2}, while, if $b < - 1$, then we similarly apply
  $\psi_s^{a, b} = \psi_s^{a, b + 2} - \psi_s^{a - 2, b + 2}$. Proceeding
  recursively, the general case reduces to the cases $(a, b)$ with $a \in \{-
  1, 0\}$ and $b \geq 0$ as well as the cases $(a, b)$ with $b \in \{- 1,
  0\}$ and $a \geq 0$. Note that, if the condition $a + b \geq 0$
  holds initially, then it holds for all the recursively generated cases.
  
  In summary, it remains to show the claim in the cases where $(a, b)$ takes
  one of the four forms $(a, 0)$, $(a, - 1)$, $(0, b)$, or $(- 1, b)$ with $a,
  b > 0$ (in each of these cases, we obviously have $a + b \geq 0$). In
  the remainder, we will show how to prove the case $(a, 0)$. The other cases
  follow similarly.
  
  Denote $D = \mathd / \mathd \tau$. Note that $D \sec^a (\tau) = a \tan
  (\tau) \sec^a (\tau)$. Differentiating once more, we find
  \begin{equation*}
    \sec^{a + 2} (\tau) = \frac{1}{a^{} (a + 1)} (D^2 + a^2) \sec^a (\tau) .
  \end{equation*}
  This shows that, for odd $a$,
  \begin{equation*}
    \sec^a (\tau) = \frac{1}{(a - 1) !} (D^2 + (a - 2)^2) (D^2 + (a - 4)^2)
     \cdots (D^2 + 1^2) \sec (\tau),
  \end{equation*}
  which implies that $\psi^{a, 0}_s$ is a linear combination of derivatives
  $D^j \psi_s^{1, 0}$ of the secant Dirichlet series; more precisely, we find
  that
  \begin{equation*}
    \psi^{a, 0}_s (\tau) = \sum_{j = 0}^{(a - 1) / 2} \frac{b_j}{\pi^{2 j}}
     (D^{2 j} \psi^{1, 0}_{s + 2 j}) (\tau)
  \end{equation*}
  for some rational numbers $b_j$. By the assumption $s \geq a + 1$,
  Proposition~\ref{prop:eval:secD} applies and the claimed evaluation of
  $\psi^{a, 0}_s (\rho)$ follows when $a > 0$ is odd. For even $a$, we
  similarly have
  \begin{equation*}
    \sec^a (\tau) = \frac{1}{(a - 1) !} (D^2 + (a - 2)^2) (D^2 + (a - 4)^2)
     \cdots (D^2 + 2^2) D \tan (\tau),
  \end{equation*}
  demonstrating that $\psi^{a, 0}_s$ now is a linear combination of
  derivatives of the tangent Dirichlet series. Our claim therefore follows
  analogously from Proposition~\ref{prop:eval:D}.
\end{proof}

\section{Conclusion}

We have shown that all Dirichlet series $\sum_{n = 1}^{\infty} f (\pi n \tau)
/ n^s$ of the appropriate parity, with $f (\tau)$ an arbitrary product of the
elementary trigonometric functions, evaluate as a (simple) algebraic multiple
of $\pi^s$ if $\tau$ is a real quadratic irrationality. Can this, in
interesting cases, be extended to series such as
\begin{equation*}
  \sum_{n = 1}^{\infty} \frac{\cot (\pi n \tau_1) \cdots \cot (\pi n
   \tau_r)}{n^s},
\end{equation*}
where $\tau_1, \ldots, \tau_r$ are quadratic (or algebraic) irrationalities?
Some examples are given in \cite[Example 6.5]{kmt-barnes}, where it is shown
that, for instance,
\begin{equation*}
  \sum_{n = 1}^{\infty} (- 1)^{n + 1} \frac{\csc (\pi n \zeta_5) \csc (\pi n
   \zeta_5^2) \cdots \csc (\pi n \zeta_5^4)}{n^6} = \frac{\pi^6}{935, 550},
\end{equation*}
with $\zeta_5 = e^{2 \pi i / 5}$.

Our method for evaluating trigonometric Dirichlet series proceeds in a
recursive way. In certain special cases, such as \cite[Theorem
5.2]{berndt-dedekindsum76}, \cite[Corollary 6.4]{kmt-barnes} or
\cite[Proposition 1]{lrr-secantzeta}, the evaluations can be made entirely
explicit. It is natural to wonder how much more explicit the evaluations given
in this paper can be made in the general case.

\begin{acknowledgements}
I thank Florian Luca for sharing his insight into
his proof \cite[Theorem 1]{lrr-secantzeta} of the convergence of the secant
Dirichlet series, and I am grateful to Bruce Berndt for helpful comments on an
early version of this paper.
\end{acknowledgements}


\end{document}